\documentclass{article}

\usepackage{amsmath}
\usepackage{amsthm}
\usepackage{amssymb}
\usepackage{mathrsfs}
\usepackage{graphicx}
\usepackage{verbatim}
\usepackage{tikz}
\usepackage[all]{xy}
\usepackage{graphics}
\usepackage[colorlinks=true]{hyperref}
\hypersetup{urlcolor=blue, citecolor=red}

\newtheorem{theorem}{Theorem}[section]
\newtheorem{lemma}[theorem]{Lemma}
\newtheorem{corollary}[theorem]{Corollary}
\newtheorem{proposition}[theorem]{Proposition}
\theoremstyle{definition}

\newtheorem*{notation}{Notation}
\newtheorem{remark}{Remark}

\newcommand{\R}{\mathbb{R}}
\newcommand{\Z}{\mathbb{Z}}
\newcommand{\N}{\mathbb{N}}

\def\beq{\begin{equation}}
\def\eeq{\end{equation}}
\def\pa{\partial}
\def\D{\Delta}
\def\f{\varphi}
\def\eps{\varepsilon}

\title{\sc Symmetry and uniqueness of nonnegative solutions of some problems in the halfspace}
\author{Alberto Farina and Nicola Soave}

\begin{document}



\maketitle
{\footnotesize
\centerline{Alberto Farina}
 \centerline{LAMFA, CNRS UMR 7352, Universit\'e de Picardie Jules Verne}
   \centerline{33 rue Saint-Leu, 80039 Amiens, France}
   \centerline{and}
   \centerline{Institut Camille Jordan, CNRS UMR 5208, Universit\'e Claude Bernard Lyon I}
   \centerline{43 boulevard du 11 novembre 1918, 69622 Villeurbane cedex, France}
   \centerline{email: alberto.farina@u-picardie.fr}

\medskip 
\centerline{Nicola Soave}
 \centerline{Universit\`a di Milano Bicocca - Dipartimento di Ma\-t\-ema\-ti\-ca e Applicazioni}
   \centerline{Via Cozzi 53}
   \centerline{20125 Milano, Italy}
\centerline{email: n.soave@campus.unimib.it}   }

\begin{abstract}
\noindent We derive some $1$-D symmetry and uniqueness or non-existence results for nonnegative solutions of
\[
\begin{cases}
-div\left(A(x)\nabla u\right)=u-g(x) & \text{in $\R^N_+$}\\
u=0 & \text{on $\pa \R^N_+$}
\end{cases}
\]
in low dimension, under suitable assumptions on $A$ and $g$. Our method is based upon a combination of Fourier series and Liouville theorems. 
\end{abstract}

\section{Introduction}

This paper concerns symmetry and uniqueness or non-existence of nonnegative solutions to problems of type
\beq\label{pb completo}
\begin{cases}
-div\left(A(x)\nabla u\right)=u-g(x) & \text{in $\R^N_+$}\\
u=0 & \text{on $\pa \R^N_+$}
\end{cases}
\eeq
in low dimension; here $div \left(A(x)\nabla \right)$ is an elliptic operator (not necessarily uniformly elliptic). As far as $g$ is concerned, we will see that the existence and the properties of nonnegative solutions to \eqref{pb completo} depend strongly on it. With this in mind, we will start considering as model problem 
\beq\label{pb modello}
\begin{cases}
-\Delta u = u -1& \text{in $\R^N_+$}\\
u=0 & \text{on $\pa \R^N_+$}.
\end{cases}
\eeq
In section \ref{sezione pb modello}, we will prove the following statement. 
\begin{theorem}\label{teo pb modello}
Let $N=2$ or $3$. If $u \in \mathcal{C}^2(\overline{\R}^N_+)$ solves problem \eqref{pb modello} and 
\beq\label{hp su u}
\forall M>0 \quad \exists C(M)>0 \quad : \quad 0 \leq u(x) \leq C(M) \quad \forall x \in \R^{N-1} \times [0,M],
\eeq
then 
\[
u(x',x_N)=1-\cos x_N.
\]
\end{theorem}
\noindent This is a result of uniqueness and of $1$-D symmetry, i.e. the (unique) solution of \eqref{pb modello} is a function depending only on $x_N$. Note that assumption \eqref{hp su u} means that $u$ is nonnegative and bounded in every strip of type $\R^{N-1} \times [0,M]$.\\
Unfortunately, we will see that the assumption "$N=2$ or $3$" is substantial for our proof. However, we can still say something for the model problem in higher dimension. This will be the object of subsection \ref{sec higher}. 

As far as the generalization towards problem \eqref{pb completo} is concerned, we will see in section \ref{section div}, Theorem \ref{teo pb div}, that the presence of $div \left(A(x)\nabla \right)$ instead of the laplacian does not affect the previous result, under suitable assumptions on $A$.\\
A further natural generalization of problem \eqref{pb modello} consists in introducing a $g$ depending only on $x_N$ instead of the constant function $1$: 
\beq\label{pb completo 1D}
\begin{cases}
-div\left(A(x)\nabla u\right)=u-g(x_N) & \text{in $\R^N_+$}\\
u=0 & \text{on $\pa \R^N_+$}
\end{cases}
\eeq
In this setting, Theorem \ref{theorem pb  complete 1D 2} is the counterpart of Theorem \ref{teo pb modello}; as an immediate corollary we have
\begin{corollary}
Let $N=2$ or $3$. Under suitable assumptions on $A$ and on $g$, if $u \in \mathcal{C}^2(\overline{\R}^N_+)$ solves \eqref{pb completo 1D} and satisfies \eqref{hp su u}, then $u$ is uniquely determined and depends only on $x_N$.
\end{corollary}

Finally, always in section \ref{section general}, we will show how to use the method developed in the previous sections in order to deal with a wider class of inhomogeneous terms (depending also on $x'$), obtaining sharp results for some particular cases; for instance, we will see that if $g=g(x')$ and there exists a solution $u$ of \eqref{pb completo} satisfying \eqref{hp su u}, then $g$ has to be constant. 

\medskip

The interest in the model problem comes from Berestycki, Caffarelli and Nirenberg: in \cite{BeCaNi} they proved that a \emph{positive and bounded} solution to \eqref{pb modello} does not exist when $N \le 3$. Their result fits in a wider study of $1$-D symmetry and monotonicity for \emph{positive and bounded} solutions to
\beq\label{pb f}
\begin{cases}
-\D u=f(u) & \text{in $\R_+^N$}\\
u=0 & \text{on $\pa \R_+^{N}$},
\end{cases} 
\eeq
with $f$ Lipschitz continuous. If $N \geq 2$ and $f(0) \geq 0$, then a positive and bounded solution is strictly increasing in the $x_N$ variable (see \cite{BeCaNi,Da}). Furthermore, always in \cite{BeCaNi}, if $N \leq 3$, $f \in \mathcal{C}^1(\R)$ and $f(0) \geq 0$, they showed that a positive and bounded solution depends only on one variable ($1$-D symmetry). Another contribution, contained in \cite{BeCaNi}, is that the monotonicity and the $1$-D symmetry hold  true for $N=2$ without any restriction on the sign of $f(0)$. The proofs of the quoted results are based on the moving planes method and on a previous result in \cite{BeCaNi1}, where it is shown that if $u$ is a positive and bounded solution of \eqref{pb f} and
\[
f(M) \leq 0 \quad \text{where} \quad M= \sup_{x \in \R^N_+} u(x),
\]
then $u$ is symmetric and monotone, and $f(M)=0$. When $f$ is a power (thus $f(0)=0$), similar results has been achieved in \cite{GiSp,Gui}. We point out that our contribution is not included in the existing literature, because we are considering \emph{nonnegative and not necessarily bounded} solutions, and because in general we are interested in the case $f(0)<0$. In such a situation the moving planes method gives just partial results, as shown by Dancer \cite{Da1}. We emphasize the fact that the difference between \emph{positive} and \emph{nonnegative} is substantial for $f(0)<0$, since in this case natural solutions are nonnegative and non-monotone, and a positive solution does not necessarily exists; this is clearly the case of the model problem \eqref{pb modello}. For all these reasons, our approach is different and it is based upon a combination of Fourier series and Liouville theorems. \\
To complete the essential bibliography for this kind of problems, we mention also the work \cite{FaVa}, where symmetry and monotonicity are obtained under weaker regularity assumptions on $f$, and an extension in dimension $4$ and $5$ is given for a wide class of nonlinearities.

\begin{notation}
We will consider problems in the half space $\R^N_+:=\R^{N-1} \times (0,+\infty)$. As usual, we will denote by $(x',x_N)$ a point of $\R^{N}_+$. \\
The symbols $\nabla'$, $div'$ or $\Delta'$ will be used respectively for the gradient, the divergence or the laplacian in $\R^{N-1}$. \\
The notation $u_j$ will be used to indicate the partial derivative of $u$ with respect to the $x_j$ variable.
For any $x \in \R^N$, for any $R>0$, we will write $B_R(x)$ to indicate the ball of centre $x$ and radius $R$. If $x=0$, we will simply write $B_R$.\\
For any $A \subset \R^N$, $\chi_A$ will denote the characteristic function of $A$.\\
We will use the notation $\left\langle \cdot,\cdot\right\rangle $ for the usual scalar product in any euclidean space.\\
Given a real valued function $v$, we denote its positive part as $v^+$.
\end{notation}

\section{The model problem}\label{sezione pb modello}

In this section we consider problem \eqref{pb modello}:
\[
\begin{cases}
-\Delta u = u -1& \text{in $\R^N_+$}\\
u=0 & \text{on $\pa \R^N_+$}.
\end{cases}
\]
We aim at proving Theorem \ref{teo pb modello}.
\begin{remark}Assumption \eqref{hp su u} says that $u$ is nonnegative in the whole $\R_+^N$ and bounded in every strip of type $\R^{N-1} \times [0,M]$ (but with arbitrary growth in the $x_N$-direction). Assumption \eqref{hp su u} is obviously satisfied if $u$ is nonnegative and bounded. 
Actually it is sufficient to assume that $u$ is nonnegative and $\nabla u$ is bounded, in order to ensure \eqref{hp su u}. Indeed for every $M>0$ we have
\begin{align*}
|u(x',x_N)| &=|u(x',x_N)-u(x',0)| \leq \sup_{\xi \in [0,x_N]} |\nabla u(x',\xi)| x_N  \\
&\leq \| \nabla u\|_{\infty} M=C(M) \qquad \forall (x',x_N) \in \R^{N-1} \times [0,M].
\end{align*}
In particular we recover the non-existence result of Berestycki, Caffarelli and Nirenberg.
\end{remark}

First we focus on the problem \eqref{pb modello} in the strip $\bar \Sigma = \R^{N-1} \times [0,2\pi]$, $N\ge2$. For every $ x' \in \R^{N-1}$, we denote by $  \tilde u(x',\cdot)$ the $2\pi$-periodic extension of $ x_N \mapsto u(x',x_N)$. In view of the smoothness of $u$, it follows that the Fourier expansion of $x_N \mapsto \tilde u(x',x_N)$, i.e. 

\beq\label{espansione}
\frac{a_0(x')}{2}+\sum_{m=1}^{+\infty} \left(a_m(x')\cos{(m x_N)} + b_m(x')\sin{(m x_N)}\right),
\eeq
where
\beq\label{coefficienti} 
\begin{split}
& a_m(x'):= \frac{1}{\pi}\int_0^{2\pi} u(x',x_N)\cos{(m x_N)}\, dx_N \qquad \forall m \geq 0,\\
& b_m(x'):= \frac{1}{\pi}\int_0^{2\pi} u(x',x_N)\sin{(m x_N)}\, dx_N \qquad \forall m \geq 1,
\end{split}
\eeq
is convergent. 


\noindent Now we determine the equations satisfied by the coefficients above.

\begin{lemma}\label{equazioni per i coefficienti di Fourier}
Let $N\ge2$. For any $m \geq 1$ we have
\begin{align}
&\Delta' a_m(x')=(m^2-1)a_m(x')+\frac{1}{\pi}\left(u_N(x',0)-u_N(x',2\pi)\right) \label{eq per a_m 0}\\
&\Delta' b_m(x')=(m^2-1)b_m(x')+\frac{m}{\pi}u(x',2\pi).\label{eq per b_m 0}
\end{align}
Also,
\[
\D' a_0(x')=2-a_0(x')+\frac{1}{\pi}\left(u_N(x',0)-u_N(x',2\pi)\right).
\]
\end{lemma}
\begin{proof}
For any $m \geq 1$ we have
\begin{align*}
\Delta' a_m(x') &= \frac{1}{\pi} \int_0^{2\pi} \Delta' u(x',x_N) \cos{(m x_N)}\,dx_N \\
&= \frac{1}{\pi} \int_0^{2\pi} \left(1-u(x',x_N)-u_{NN}(x',x_N)\right) \cos{(m x_N)}\, dx_N \\
& = -\frac{1}{\pi} \int_0^{2\pi} \left(u(x',x_N) + u_{NN}(x',x_N)\right)\cos{(m x_N)}\,dx_N.
\end{align*}
Integrating by parts twice the last term we obtain
\begin{align*}
\Delta' a_m(x') = & -\frac{1}{\pi} \int_0^{2\pi} \left(u(x',x_N) \cos{(m x_N)}+ m u_{N}(x',x_N)\sin{(m x_N)} \right)\,dx_N\\
& - \frac{1}{\pi}\left[u_N(x',x_N)\cos{(m x_N)}\right]_{x_N=0}^{2\pi}\\
= & \frac{m^2-1}{\pi} \int_0^{2\pi} u(x',x_N) \cos{(m x_N)}\,dx_N \\
&- \frac{1}{\pi}\left(u_N(x',2\pi)-u_N(x',0)\right),
\end{align*} 
which is equation \eqref{eq per a_m 0}.\\
With the same procedure we can find equation \eqref{eq per b_m 0}: for any $m \geq 1$
\begin{align*}
\Delta' b_m(x')  = &\frac{1}{\pi} \int_0^{2\pi} \Delta' u(x',x_N) \sin{(m x_N)}\,dx_N \\
= &\frac{1}{\pi} \int_0^{2\pi} \left(1-u(x',x_N)-u_{NN}(x',x_N)\right) \sin{(m x_N)}\, dx_N \\
= &\frac{1}{\pi} \int_0^{2\pi} \left(-u(x',x_N)\sin{(m x_N)} + m u_{N}(x',x_N)\cos{(m x_N)}\right)\,dx_N \\
= &\frac{m^2-1}{\pi} \int_0^{2\pi} u(x',x_N) \sin{(m x_N)}\,dx_N \\
&+ \frac{m}{\pi}\left[u(x',x_N)\cos{(m x_N)}\right]_{x_N=0}^{2\pi}. 
\end{align*}
As far as $a_0$ is concerned, we have
\begin{align*}
\Delta' a_0(x') &=\frac{1}{\pi}\int_0^{2\pi} \Delta' u(x',x_N) \, dx_N \\
& = \frac{1}{\pi} \int_0^{2\pi} \left(1-u(x',x_N)-u_{NN}(x',x_N)\right)\,dx_N \\
& = 2 - \frac{1}{\pi}\int_0^{2\pi}u(x',x_N)\,dx_N + \frac{1}{\pi} \left(u_N(x',0)-u_N(x',2\pi)\right). \qedhere
\end{align*}
\end{proof}

\begin{lemma}
Both $b_1$ and $a_1$ are constant; moreover, 
\[
u(x',2\pi) =0, \quad u_N(x',0)=0  \quad \text{and} \quad u_N(x',2\pi)=0 \quad \forall x' \in \R^{N-1}.
\]
\end{lemma}
\begin{proof}
Using \eqref{eq per b_m 0} with $m=1$ we have 
\[
\Delta' b_1(x')=\frac{1}{\pi}u(x',2\pi)\geq 0.
\]
Therefore, thanks to \eqref{hp su u}, $b_1$ is a subharmonic and bounded function in $\R^{N-1}$ with $N=2$ or $3$; the Liouville theorem for subharmonic functions implies that it is constant, so that in particular $\D' b_1 \equiv 0$, i.e. $u(x',2\pi)=0$ for every $x' \in \R^{N-1}$.\\
Note that, since $u\geq0$ and $u(x',2\pi)=0$, each $(x',2\pi)$ is a point of minimum for $u$; consequently $u_N(x',2\pi)=0$, and this makes possible to prove that also $a_1$ is constant: indeed 
\[
\Delta' a_1(x')=\frac{1}{\pi}u_N(x',0) .
\]
Since $u(x',0)=0$ and $u\geq 0$ in $\Sigma$, it follows that $u_N(x',0) \geq 0$ for every $x' \in \R^{N-1}$. Hence $a_1$ is a subharmonic and bounded function in $\R$ or $\R^2$, and has has to be constant. It follows in particular that $u_N(x',0)\equiv 0$ in $\R^{N-1}$.
\end{proof}

An important consequence of the previous Lemma is that the equations for $a_m$ and $b_m$ simplify as
\begin{align}
&\Delta' a_m(x')=(m^2-1)a_m(x') \qquad \forall m \geq 2 \label{eq per a_m 1}\\
&\Delta' b_m(x') = (m^2-1) b_m(x') \qquad \forall m \geq 2 \label{eq per b_m 1} .
\end{align}
Hence, for $m \geq 2$, $a_m$ and $b_m$ satisfy an equation of type
\beq\label{eq per m>1}
-\Delta' v(x')+\lambda v(x')=0 \qquad \text{in $\R^{N-1}$},
\eeq
with $\lambda>0$. We point out that both $a_m$ and $b_m$ are bounded in absolute value in $\Sigma$ (this follows from assumption \eqref{hp su u}). \\
Bounded solutions of  \eqref{eq per m>1}  has to vanish identically. This is an immediate  consequence of the following general result.

\begin{lemma}\label{coefficienti nulli} 
Assume $N\ge2$ and let $v \in \mathcal{C}^2(\R^{N-1})$ be a subsolution of 
\beq\label{diseq per m>1}
-\Delta' v(x')+c(x')v(x')\le 0 \qquad \text {in $\R^{N-1}$},
\eeq
with $c(x') \ge \lambda > 0$ {in $\R^{N-1}$}. \\
If $v^+$ has at most algebraic growth at infinity, then $v \le 0$ {in $\R^{N-1}$},
\end{lemma}

For the proof, it will be useful the following Lemma.

\begin{lemma}\label{lemma tecnico}
Let $\vartheta>0$, $\gamma>0$, be such that $\vartheta<2^{-\gamma}$. Let $R_0>0$, $C>0$ and $I:(R_0,+\infty) \to [0,+\infty)$ be such that
\beq\label{eq per I}
\begin{cases}
I(R)\leq \vartheta I(2R) & \forall R>R_0\\
I(R) \leq CR^\gamma & \forall R>R_0.
\end{cases}
\eeq
Then $I(R)=0$ for every $R>R_0$.
\end{lemma}
\begin{proof}
Iterating the first one of \eqref{eq per I} we obtain, for every $k \in \N$,
\[
I(R) \leq \vartheta^k I(2^k R) \qquad \forall R>R_0.
\]
Now the second one gives
\[
I(R) \leq C \left(\vartheta 2^\gamma\right)^k R^\gamma \qquad \forall R>R_0,\  \forall k \in \N.
\]
Since $0<\vartheta 2^\gamma <1$, letting $k \to \infty$ we obtain $I(R) \leq 0$ for $R>R_0$.
\end{proof}


\begin{proof}[Proof of Lemma \ref{coefficienti nulli}] 
We introduce a $\mathcal{C}^\infty$ cut-off function $\f: [0,+\infty) \to \R$ such that
\[
\begin{cases}
\f(t)=1 & t \in [0,1] \\
\f(t)=0 & t \in [2,+\infty)\\
0 \leq \f(t) \leq 1 & t \in (1,2).
\end{cases} 
\]
We set, for every $R>0$, $\f_R(x'):=\f(|x'|/R)$, which is defined on $\R^{N-1}$. Hence 
\[
\nabla' \f_R(x')= \frac{x'}{R|x'|}\f'\left(\frac{|x'|}{R}\right).
\]
In particular
\beq\label{magg su grad phi}
\left| \nabla ' \f_R(x') \right| \leq \frac{C}{R} \chi_{B_{2R}}(x')
\eeq
where $C$ is a constant independent of $R$.\\
Testing \eqref{diseq per m>1}  with $v^+ \f_R^2$ we get 
\begin{multline}\label{eq524}
\int_{\R^{N-1}} \left(|\nabla' v^+|^2 + \lambda \left(v^+\right)^2\right) \f_R^2 \le  
\int_{\R^{N-1}} \left(|\nabla' v^+|^2 +  c\left(v^+\right)^2\right) \f_R^2  \\
\le -2 \int_{\R^{N-1}} v^+ \f_R \left\langle \nabla' v^+,\nabla' \f_R\right\rangle \leq 2\int_{\R^{N-1}} v^+ \f_R \left|\left\langle \nabla' v^+,\nabla' \f_R\right\rangle \right|. 
\end{multline}
We can use the Cauchy-Schwarz and the Young inequalities: for every $\eps>0$ there exists $C_\eps>0$ such that
\begin{multline*}
2\int_{\R^{N-1}} v^+ \f_R \left|\left\langle \nabla' v^+,\nabla' \f_R\right\rangle \right| \leq 2 \int_{\R^{N-1}} v^+ \f_R |\nabla' v^+| |\nabla' \f_R| \\
\le 2\eps \int_{\R^{N-1}} \f_R^2 |\nabla' v^+|^2 + 2C_\eps \int_{\R^{N-1}} \left(v^+\right)^2 \left|\nabla' \f_R\right|^2.
\end{multline*}
Coming back to \eqref{eq524}, we obtain
\[
\int_{\R^{N-1}} \left( (1-2\eps) |\nabla' v^+|^2+ \lambda \left(v^+\right)^2 \right) \f_R^2 \leq 2C_\eps \int_{\R^{N-1}} \left(v^+\right)^2 |\nabla' \f_R|^2.
\]
Choosing $\eps < 1/2$ and using the \eqref{magg su grad phi}, we deduce
\[
\int_{B_R} \left(v^+\right)^2 \leq \frac{C}{\lambda R^2} \int_{B_{2R}} \left(v^+\right)^2.
\]
Also, since $v^+$ has at most algebraic growth at infinity,  we have for any $R > 1$
\[
\int_{B_R} (v^+)^2 \leq C' R^{N + 2k}
\]
for some $k\ge 0$, $C'>0$ independent of $R$.\\
We are in position to apply Lemma \ref{lemma tecnico}, with 
\[
I(R):=\int_{B_R} (v^+)^2.
\]
Here $\gamma=N+2k$; note that there exists $R_0>1$ such that $\frac{C'}{\lambda R^2} 
<2^{-N-2k}$ for every $R \ge R_0$. We set $\vartheta=\frac{C'}{\lambda R_0^2} $ and we apply Lemma \ref{lemma tecnico} to obtain 
\[ 
\int_{B_R} (v^+)^2 = 0 \qquad \forall R>R_0 \Rightarrow v^+ \equiv 0. \qedhere
\]
\end{proof}

\medskip
\begin{proof}[Conclusion of the proof of Theorem \ref{teo pb modello}]
Applying Lemma \ref{coefficienti nulli} to equations \eqref{eq per a_m 1} and \eqref{eq per b_m 1},  we have that the Fourier coefficients $a_m$ and $b_m$ are identically $0$ for any $m \geq 2$. Hence, the Fourier series \eqref{espansione} is reduced to
\beq\label{forma di u}
\frac{a_0(x')}{2}+a_1 \cos{x_N} + b_1 \sin{x_N}
\eeq
and, for every $x \in \bar \Sigma$, it is equal to $u(x)$.  
 
The initial condition $u(x',0)=0$ reads
\[
\frac{a_0(x')}{2}+a_1=0 \Rightarrow \text{$a_0$ is constant, equal to $-2 a_1$}.
\]
We also proved that $u_N(x',0)=0$, which implies $b_1=0$.\\
Plugging the expression of $u$ inside the equation $-\D u =u-1$ we obtain
\[
-a_1 \cos{x_N}+\frac{a_0}{2} +a_1 \cos{x_N}-1=0 \Rightarrow a_0=2,
\]
and hence $a_1 = -1$.\\
We proved that if $u \in \mathcal{C}^2(\overline{\R}_+^N)$ is a solution of \eqref{pb modello} satisfying \eqref{hp su u} then $u(x',x_N)=1-\cos{x_N}$ in $\bar \Sigma= \R^{N-1}\times [0,2\pi]$. \\
To extend the result in the whole $\R^N_+$ we set
\[
v_1(x',x_N):=u(x',x_N+2\pi).
\]
It is straightforward to check that $v_1$ is a nonnegative solution of \eqref{pb modello} and satisfies \eqref{hp su u}, so that it has to coincide with $1-\cos{x_N}$ in $\bar \Sigma$; this means that $u(x',x_N)=1-\cos{x_N}$ for $(x',x_N) \in \R^{N-1} \times [0,4\pi]$. The thesis follows by iteration of this argument.
\end{proof}

\subsection{The model problem in higher dimension}\label{sec higher}

In our proof it was crucial the possibility of applying the Liouville Theorem for subharmonic functions, which holds only in $\R$ and $\R^2$. Therefore, despite the fact that our statement seems to be natural in any dimension, we cannot prove it. However, it is still possible to collect some properties of any solution of problem \eqref{pb modello} satisfying \eqref{hp su u} for $N \geq 4$.\\
We can focus again on the problem in the strip $\bar \Sigma$, developing $u$ (or, better, its $2\pi$-periodic extension $\widetilde{u}$ in the $x_N$ variable) in formal Fourier series with respect to $x_N$. Note that Lemma \ref{equazioni per i coefficienti di Fourier} still holds true. Now, in our analysis the key properties of the solutions was 
\beq\label{passo base 1}
u(x',2\pi) \equiv 0 \quad \text{and} \quad u_N(x',0) \equiv 0 \quad \text{in $\R^{N-1}$}.
\eeq 
In this way, equations \eqref{eq per a_m 0} and \eqref{eq per b_m 0} are considerably simplified, since all the boundary terms have to vanish identically. This permits to get Theorem \ref{teo pb modello}.

\begin{proposition}\label{prop large N}
Let $N \geq 2$. Let $u \in \mathcal{C}^2(\overline{\R}^N_+)$ be a solution of problem \eqref{pb modello} which satisfies \eqref{hp su u}. Let $a_m$ and $b_m$ its formal Fourier coefficients, defined by \eqref{coefficienti}. Assume \eqref{passo base 1} holds true. Then
\[
u(x',x_N)=1-\cos x_N.
\]
\end{proposition}
\begin{proof}
Under assumption \eqref{passo base 1}, equations \eqref{eq per a_m 0} and \eqref{eq per b_m 0} are reduced to
\begin{align*}
&\D' a_m(x')=(m^2-1)a_m(x') \qquad \forall m \ge 1 \\
&\D' b_m(x')=(m^2-1)b_m(x') \qquad \forall m \ge 1
\end{align*}
(note that $u_N(x',2\pi)=0$ since $u(x',2\pi)=0$ and $u \geq 0$). Hence, Lemma \ref{coefficienti nulli} implies that $a_m \equiv 0 \equiv b_m$ for every $m \ge 2$, while from the classical Liouville theorem for harmonic function it follows that $a_1$ and $b_1$ are constant, so that 
\[
u(x',x_N)= \frac{a_0(x')}{2}+ a_1 \cos x_N + b_1 \sin x_N \qquad \text{in $\Sigma$}.
\]
Now we can conclude as in the proof of Theorem \ref{teo pb modello}.
\end{proof}

Also if we cannot prove \eqref{passo base 1}, it is possible to deduce something for the formal Fourier coefficients.

\begin{proposition}
Let $N \ge 2$. Let $u \in \mathcal{C}^2(\overline{\R}^N_+)$ be a solution of problem \eqref{pb modello} which satisfies \eqref{hp su u}. Let $a_m$ and $b_m$ its formal Fourier coefficients, defined by \eqref{coefficienti}. Then
\begin{itemize}
\item[($i$)] $b_m \leq 0$ for every $m \geq 2$.
\item[($ii$)] $\frac{b_n}{n} \geq \frac{b_m}{m}$ for every $n>m \geq 2$.
\item[($iii$)] for every $m \geq 2$, either $b_m<0$ or $b_m \equiv 0$ in $\R^{N-1}$.
\end{itemize}
\end{proposition}
\begin{proof}
($i$) For every $m \ge 2$ we have 
\beq\label{eq524*}
-\D' b_m(x')+(m^2-1)b_m(x')=-\frac{m}{\pi}u(x',2\pi) \leq 0.
\eeq
From Lemma \ref{coefficienti nulli}, \emph{which holds true in any dimension}, we deduce that 
\beq\label{eq525}
b_m \leq 0 \qquad \forall m \geq 2.
\eeq
($ii$) For $m \geq 2$, let us divide equation \eqref{eq per b_m 0} by $m$:
\[
\frac{\D' b_m(x')}{m}=\frac{m^2-1}{m}b_m(x')+\frac{1}{\pi}u(x',2\pi).
\]
If $n>m\geq 2$
\begin{multline*}
-\D'\left(\frac{b_m(x')}{m}-\frac{b_n(x')}{n}\right)+\left(n^2-1\right)\left(\frac{b_m(x')}{m}-\frac{b_n(x')}{n}\right) \\
= \left(n^2-m^2\right) \frac{b_m(x')}{m} \leq 0,
\end{multline*}
thanks to the fact that $b_m\leq 0$. Again, by means of Lemma \ref{coefficienti nulli}, we get
\[
\frac{b_n}{n}\geq \frac{b_m}{m} \qquad \forall \ n>m \geq 2.
\]
($iii$) if there exists $\bar x' \in \R^{N-1}$ such that $b_{m}(\bar x')=0$, the strong maximum principle implies $b_{m} \equiv 0$.
\end{proof}

It is particularly interesting to observe that, if we know that one $b_m$ vanishes in one point of $\R^{N-1}$, then we can recover Theorem \ref{teo pb modello}.

\begin{corollary}
Let $N \ge 2$. Let $u \in \mathcal{C}^2(\overline{\R}^N_+)$ be a solution of problem \eqref{pb modello} satisfying \eqref{hp su u}. Let $a_m$ and $b_m$ its formal Fourier coefficients, defined by \eqref{coefficienti}. \\
If there exist $\bar m \geq 2$ and $\bar x' \in \R^{N-1}$ such that $b_{\bar m}(\bar x') = 0$, then $u(x',x_N)=1-\cos x_N$.
\end{corollary}
\begin{proof}
By point ($iii$) of the previous Proposition we know that $b_{\bar m} \equiv 0$. \\
Hence, from \eqref{eq per b_m 0} for $\bar m$ we get 
\[
u(x',2\pi) \equiv 0 \qquad \text{in $\R^{N-1}$}.
\]
As a consequence
\[
-\D' b_m(x')+(m^2-1) b_m(x')=0 \qquad \forall m \ge 1,
\]
which implies through Lemma \ref{coefficienti nulli} that $b_m \equiv 0$ for every $m \ge 2$; also, $b_1$ turns out to be a bounded harmonic function on the whole $\R^{N-1}$, so that it has to be constant. Now we show that $b_1 =0$. Note that 
\[
\widetilde{u}(x',x_N)- b_1 \sin x_N=\frac{a_0(x')}{2}+ \sum_{m=0}^{+\infty} a_m(x') \cos(m x_N);
\]
hence, $w(x',x_N)=\widetilde{u}(x',x_N)-b_1 \sin x_N$ is an even $2\pi$-periodic function in the $x_N$ variable. Since we are assuming $u \in \mathcal{C}^2(\overline{\R}^N_+)$ and $u(x',2\pi) =0$, the function $w$ is continuous on the whole $\R$, and has continuous derivative with respect to $x_N$, except at most in $(x',0+2k\pi)$, with $k \in \Z$. However, the right and left derivatives in these points exist, and in particular
\[
w_N(x',2\pi^-)= u_N(x',2\pi^-)-b_1=-b_1.
\]
By periodicity and oddness of $w_N$ it results
\[
b_1=w_N(x',0^+)=u_N(x',0^+)-b_1=u_N(x',0)-b_1 \Rightarrow u_N(x',0)=2b_1.
\]
Note that $u_N(x',0)$ is constant. Now, plugging this expression in equation \eqref{eq per a_m 0} with $m=1$ we obtain
\[
\D' a_1(x')=\frac{2b_1}{\pi} \Rightarrow \D'\left(a_1(x')-\frac{b_1 x_1^2}{\pi}\right)=0:
\]
the function $a_1(x')- b_1 x_1^2/\pi$ is harmonic in the whole $\R^N$ and has at most algebraic growth with rate $2$ (since $a_1$ is bounded): therefore, the Liouville Theorem implies that 
\[
a_1(x')= \frac{b_1 x_1^2}{\pi} + P(x'),
\]
where $P$ is a harmonic polynomial. To sum up, $a_1$ is a bounded polynomial, thus it is constant, which in turns gives $\D' a_1 =0$, i.e. $b_1=0$ and finally $u_N(x',0)=0$. The thesis follows now from Proposition \ref{prop large N}.
\end{proof}

\section{More general operators}\label{section div}

In this section we generalize the approach adopted for the model problem to a more general family of elliptic equations (not necessarily uniformly elliptic) obtained by substituting the laplacian with a class of operators in divergence form. To be precise, let $A(x')$ be a $N \times N$ matrix of type
\beq\label{hp 1 su A}
A(x')= \left(\begin{array}{cc}
      \widehat{A}(x') & 0\\
      0 & 1
      \end{array}\right),
\eeq
where $\widehat{A}(x')$ is a $(N-1) \times (N-1)$ symmetric and real matrix with entries  $a_{ij} \in C^1\left(\R^{N-1}\right) \cap L^\infty\left(\R^{N-1}\right)$, such that
\beq\label{hp 2 su A}
\forall x' \in \R^{N-1}, \forall \xi \in \R^{N-1} \setminus \{0\}: \ \sum_{i,j=1}^{N-1} a_{ij}(x')\xi_i \xi_j > 0.
\eeq
Of course, if $N=2$ then $\widehat{A}(x')$ is a scalar positive function.\\
For the reader's convenience, we recall the following generalization of the classical Liouville theorem, see \cite{Fa}.
 
\begin{theorem}\label{teo Liouville div}
Let $q \geq 0$ and $B(x)=(b_{ij}(x))$ be a symmetric real matrix, whose entries are $L^{\infty}\left(\R^2\right)$ functions satisfying:
\[
\text{for a.e. } x' \in \R^2, \forall \xi \in \R^2 \setminus \{0\}: \ \sum_{i,j=1}^{2} b_{ij}(x')\xi_i \xi_j > 0.
\]
Let $v \in H^1_{loc}\left(\R^2\right)$ be a distribution solution of 
\[
\begin{cases}
-div\left(B(x')\nabla v\right)+q\left( \left\langle B(x')\nabla v, \nabla v\right\rangle  \right) \geq 0 & \text{in $\R^2$} \\
v(x') \geq -C & \text{a.e. in $\R^2$},
\end{cases}
\]
for some positive constant $C$. Then $v$ is a constant function.
\end{theorem}

\begin{remark}
It is not difficult to pass from the two dimensional case to the scalar case: it is sufficient to consider a solution of an ODE as a solution of the corresponding PDE. 
\end{remark}

These results enable us to try to use the arguments of section \ref{sezione pb modello} for the study of
\beq\label{pb div}
\begin{cases}
- div\left(A(x)\nabla u\right)= u-1& \text{in $\R^N_+$}\\
u=0 & \text{on $\pa \R^N_+$}.
\end{cases}
\eeq
Note that, due to the particular form of $A$ (cfr. equation \eqref{hp 1 su A}), it is reasonable to think that \eqref{pb div} inherits the structure of the model problem solved in the previous section. It is also immediate to check that the function $1-\cos{x_N}$ is, again, a nonnegative solution of \eqref{pb div} satisfying \eqref{hp su u}. We plan to prove that it is also unique in this class for $N=2$ and $3$. 

\begin{theorem}\label{teo pb div}
Let $N=2$ or $3$. Let $A(x')$ be a $N \times N$ matrix of type \eqref{hp 1 su A}, where $\widehat{A}(x')$ is a $(N-1) \times (N-1)$ symmetric and real matrix with entries  $a_{ij} \in C^1\left(\R^{N-1}\right) \cap L^\infty\left(\R^{N-1}\right)$, and such that \eqref{hp 2 su A} holds true. If $u \in \mathcal{C}^2(\overline{\R}^N_+)$ solves \eqref{pb div} and satisfies \eqref{hp su u}, then 
\[
u(x',x_N)=1-\cos x_N.
\]
\end{theorem}

\noindent We can follow the proof of Theorem \ref{teo pb modello}. Again, $\Sigma=\R^{N-1}\times (0,2\pi)$. For every $ x' \in \R^{N-1}$, we denote by $  \tilde u(x',\cdot)$ the $2\pi$-periodic extension of $ x_N \mapsto u(x',x_N)$. As for the model problem, from the smoothness of $u$ it follows that the Fourier expansion of $x_N \mapsto \tilde u(x',x_N)$ is convergent:
\[
\widetilde{u}(x',x_N)=\frac{a_0(x')}{2}+\sum_{m=1}^{+\infty} \left(a_m(x')\cos{(m x_N)} + b_m(x')\sin{(m x_N)}\right),
\]
where $a_m$ and $b_m$ are defined by \eqref{coefficienti}. 

With a slightly modification of the proof of Lemma \ref{equazioni per i coefficienti di Fourier}, we obtain

\begin{lemma} Let $N\ge2$. 
For any $m \geq 1$ we have
\begin{align}
& div'\left(\widehat{A}(x') \nabla' a_m(x')\right)=(m^2-1)a_m(x')+\frac{1}{\pi} \left(u_N(x',0)-u_N(x',2\pi) \right)\label{eq a_m 0 div}\\
& div'\left(\widehat{A}(x') \nabla' b_m(x')\right)=(m^2-1)b_m(x') + \frac{m}{\pi}u(x',2\pi). \label{eq b_m div}
\end{align}
Also,
\beq\label{eq a_0 div}
div'\left(\widehat{A}(x') \nabla' a_0(x')\right)= 2 -a_0(x')+\frac{1}{\pi} \left(u_N(x',0)-u_N(x',2\pi)\right).
\eeq
\end{lemma}
\begin{proof}
For any $m \geq 1$ we have
\beq\label{eq2}
div'\left(\widehat{A}(x') \nabla' a_m(x')\right)  = \frac{1}{\pi} \int_0^{2\pi} div'\left(\widehat{A}(x') \nabla 'u(x',x_N)\right) \cos{(m x_N)}\,dx_N .
\eeq
Since $a_{iN}=a_{Nj}\equiv 0$ for any $i,j \neq N$, we have
\beq\label{eq3}
\begin{split}
div'\left(\widehat{A}(x')\nabla' u(x',x_N)\right) &= \sum_{i=1}^{N-1} \pa_i \left(\sum_{j=1}^{N-1} a_{ij}(x') u_j(x',x_N)\right)\\
&=\sum_{i=1}^N \pa_i \left(\sum_{j=1}^N a_{ij}(x') u_j(x',x_N)\right)-u_{NN}(x',x_N)\\
&= 1-u(x',x_N)-u_{NN}(x',x_N).
\end{split}
\eeq
Hence equation \eqref{eq2} becomes
\[
div'\left(\widehat{A}(x')\nabla' a_m(x')\right)  = \frac{1}{\pi} \int_0^{2\pi} \left(1-u(x',x_N)-u_{NN}(x',x_N)\right)\cos{( m x_N)}\,dx_N.
\]
Now, as usual, we can integrate by parts twice the last term and pass to
\[
div'\left(\widehat{A}(x')\nabla' a_m(x')\right)  = (m^2-1)a_m(x')+ \frac{1}{\pi}\left(u_N(x',0)-u_N(x',2\pi)\right),
\]
which is \eqref{eq a_m 0 div}. The same procedure gives \eqref{eq b_m div} and \eqref{eq a_0 div}.
\end{proof}

\begin{lemma}\label{lemma on a_1 and b_1 div}
Both $b_1$ and $a_1$ are constant. Moreover, 
\[
u(x',2\pi)=0, \quad u_N(x',2\pi)=0 \quad \text{and} \quad u_N(x',0)=0 \quad \forall x' \in \R^{N-1}.
\]
\end{lemma}
\begin{proof}
In light of the previous Lemma, we have
\[
div'\left(\widehat{A}(x')\nabla' b_1(x')\right)= \frac{1}{\pi}u(x',2\pi) \geq 0;
\]
the function $b_1$ is bounded (since $u$ satisfies \eqref{hp su u}), and since $N=2$ or $3$ we are in position to apply Theorem \ref{teo Liouville div}:
\[
b_1 = const. \Rightarrow u(x',2\pi)=\pi div'\left(\widehat{A}(x')\nabla' b_1(x')\right) =0.
\]
Note that now $u_N(x',2\pi)=0$, since $(x',2\pi)$ is a point of minimum of $u$ for every $x' \in \R^{N-1}$. Therefore, equation \eqref{eq a_m 0 div} becomes
\[
div'\left(\widehat{A}(x')\nabla' a_1(x')\right)=\frac{1}{\pi}u_N(x',0) \geq 0;
\]
this means that $a_1$ satisfies the assumptions of Theorem \ref{teo Liouville div}: 
\[
a_1= const. \Rightarrow u_N(x',0)= \pi div'\left(\widehat{A}(x')\nabla' a_1(x')\right) = 0.    \qedhere
\]
\end{proof}

As a consequence, the equations for $a_m$ and $b_m$ simplify: \begin{align}
&div'\left(\widehat{A}\nabla' b_m(x')\right) = (m^2-1) b_m(x') \qquad \forall m \geq 2 \label{eq per b_m} \\
&div'\left(\widehat{A}\nabla' a_m(x')\right) =(m^2-1)a_m(x') \qquad \forall m \geq 2. \label{eq for a_m}
\end{align}
In this way, we proved that for any $m \geq 2$ both the coefficients $a_m$ and $b_m$ are bounded solution of an equation of type
\beq\label{eq per coefficienti div}
-div'\left(\widehat{A}(x') \nabla' v(x')\right)+\lambda v(x')=0 \qquad \text{in $\R^{N-1}$},
\eeq
with $\lambda>0$. In analogy with the model problem, we state the following result.

\begin{lemma}\label{coefficienti nulli div}
Assume $N\ge 2$ and let $v \in \mathcal{C}^2\left(\R^{N-1}\right)$ a subsolution of
\[
-div'\left(\widehat{A}(x') \nabla' v(x')\right)+c(x') v(x') \le 0 \qquad \text{in $\R^{N-1}$}
\]
with $c(x') \geq \lambda >0$. Here $\widehat{A}(x')$ is an $(N-1) \times (N-1)$ matrix with entries $a_{ij}$ in $L^{\infty}(\R^{N-1})$ and such that \eqref{hp 2 su A} holds true. \\
If $v^+$ has at most algebraic growth at infinity, then $v\le 0$.
\end{lemma}
\begin{proof}
For any $R>0$, let $\f_R$ be as in the proof of Lemma \ref{coefficienti nulli}. Recall the \eqref{magg su grad phi}:
\[
\left| \nabla ' \f_R(x') \right| \leq \frac{C}{R} \chi_{B_{2R}}(x').
\] 
Let us test equation \eqref{eq per coefficienti div} with $v^+ \f_R^2$:
\begin{multline}\label{eq4}
\int_{\R^{N-1}} \left(\langle \widehat{A}(x') \nabla' v^+, \nabla' v^+\rangle +c(x') \left(v^+\right)^2\right)\f_R^2 \\
= -2\int_{\R^{N-1}} v^+\f_R \langle \widehat{A}(x')\nabla' v, \nabla' \f_R\rangle .
\end{multline}
Under our assumptions on $\widehat{A}$, for almost every $x' \in \R^{N-1}$ the function
\[
(\xi_1,\xi_2) \in \R^{2(N-1)} \mapsto \langle \widehat{A}(x')\xi_1,\xi_2\rangle \in \R
\]
defines a bilinear symmetric positive definite form, so that in particular the Cauchy-Schwarz inequality holds true. Hence, using also the Young inequality, we can control the right hand side: for any $\varepsilon>0$ there exists $C_\varepsilon>0$ such that
\begin{align*}
-2 &\int_{\R^{N-1}} v^+  \f_R \langle \widehat{A}(x')\nabla' v^+, \nabla' \f_R \rangle  \leq 2\int_{\R^{N-1}} v^+ \f_R \left|\langle \widehat{A}(x')\nabla' v^+, \nabla' \f_R \rangle \right| \\
& \leq 2 \int_{\R^{N-1}} v^+ \f_R \sqrt{\langle \widehat{A}(x') \nabla' v^+, \nabla' v^+\rangle }
\sqrt{\langle \widehat{A}(x') \nabla' \f_R, \nabla' \f_R\rangle } \\
&\leq 2\varepsilon \int_{\R^{N-1}} \f_R^2 \langle \widehat{A}(x') \nabla' v^+, \nabla' v^+\rangle + 2C_\varepsilon \int_{\R^{N-1}} \left(v^+\right)^2 \langle \widehat{A}(x') \nabla' \f_R, \nabla' \f_R \rangle.
\end{align*}
Coming back to equation \eqref{eq4}, using also the fact that $c(x') \geq \lambda$, we find
\begin{multline*}
\int_{\R^{N-1}} \left(\left(1-2\varepsilon\right)\langle \widehat{A}(x') \nabla' v^+, \nabla' v^+\rangle+\lambda \left(v^+\right)^2\right)\f_R^2 \\
\leq  2C_\varepsilon \int_{\R^{N-1}} \left(v^+\right)^2 \langle \widehat{A}(x') \nabla' \f_R, \nabla' \f_R\rangle.
\end{multline*}
Choosing $\varepsilon <1/2$, using the assumptions of $\widehat{A}$ and the estimate \eqref{magg su grad phi}, we deduce
\[
\int_{B_R} \left(v^+\right)^2 \leq \frac{C}{\lambda R^2}\int_{B_{2R}} \left(v^+\right)^2.
\]
Also, since $v^+$ has at most algebraic growth, we have
\[
\int_{B_R} \left(v^+\right)^2 \le C' R^{N+2k}.
\]
We can apply Lemma \ref{lemma tecnico} again, to find
\[
\int_{B_R} \left(v^+\right)^2=0 \qquad \forall R>R_0 \text{ sufficiently large},
\]
which implies $v^+ \equiv 0$.
\end{proof}

\begin{proof}[Conclusion of the proof of Theorem \ref{teo pb div}]
The previous Lemma implies that \\$a_m \equiv 0$ and $b_m \equiv 0$ for every $m \geq 2$. Therefore, a solution $u$ of \eqref{pb div} which satisfies \eqref{hp su u} has the following  expansion in $\bar \Sigma$:
\[
u(x',x_N)=\frac{a_0(x')}{2}+a_1 \cos{x_N} + b_1 \sin{x_N},
\]
which is exactly \eqref{forma di u}. Moreover, we showed that
\[
u(x',0)=0 \quad \text{and} \quad u_N(x',0)=0 \quad \forall x' \in \R^{N-1},
\]
hence we can repeat step by step the conclusion of the proof of Theorem \ref{teo pb modello}.
\end{proof}

\section{More general problems}\label{section general}

In this section we apply the previous method to study and classify solutions of
\[
\begin{cases}
-div\left(A(x')\nabla u\right)=u-g(x',x_N) & \text{in $\R^N_+$}\\
u=0 & \text{on $\pa \R^N_+$},
\end{cases}
\]
satisfying \eqref{hp su u}, when $N=2$ or $3$. The situation here is much more involved than the one in the previous sections. Indeed, we have to face the occurrence of various phenomena such as: non-existence of solutions and/or the existence and the multiplicity of solutions. Moreover, a solution might not be a function of the $x_N$ variable only (in fact, if $g$ depends on $x'$ such a result cannot be expected). The results that we shall prove, will strongly depend on the form of the function $g$.

In what follows, we will always assume that the matrix $A$ satisfies the assumptions already imposed in the previous section. Therefore, we do not write explicitly these assumptions anymore. Since we are interested in classical solutions, we assume that $g \in \mathcal{C}(\overline{\R}_+^{N})$. We can consider the $2\pi$-periodic extension of $x_N \in (0,2\pi) \mapsto g(x',x_N)$: inside $\Sigma$ we have the expansion
\[
\widetilde{g}(x',x_N)= \frac{c_0(x')}{2}+\sum_{m=1}^\infty \left(c_m(x') \cos{(m x_N)} + d_m(x') \sin{(m x_N)}\right),
\]
where
\beq\label{coeff of g}
\begin{split}
& c_m(x'):= \frac{1}{\pi} \int_0^{2\pi}  g(x',x_N)\cos{(m x_N)} \, dx_N \qquad \forall m \geq 0 \\
& d_m(x'):= \frac{1}{\pi} \int_{0}^{2\pi}  g(x',x_N)\sin{(m x_N)} \, dx_N \qquad \forall m \geq 1.
\end{split}
\eeq
Let us define again $a_m$ and $b_m$ by \eqref{coefficienti}; these are the formal Fourier coefficients of $u$ with respect to the $x_N$-variable in $\Sigma$. We start writing down the equations satisfied by $a_m$ and $b_m$.

\begin{lemma}
For any $m \geq 0$ it results
\beq\label{eq a_m completo}
div'\left(\widehat{A}(x') \nabla' a_m(x')\right) = \left(m^2-1\right) a_m(x') + c_m(x')+\frac{1}{\pi}\left(u_N(x',0)-u_N(x',2\pi)\right);
\eeq
For any $m \geq 1$ it results
\beq\label{eq b_m completo}
div'\left(\widehat{A}(x') \nabla' b_m(x')\right) = \left(m^2-1\right) b_m(x') + d_m(x')+\frac{m}{\pi} u(x',2\pi).
\eeq
\end{lemma}
\begin{proof}
For any $m \geq 1$:
\begin{multline*}
div'\left(\widehat{A}(x') \nabla' a_m(x')\right) = \frac{1}{\pi} \int_0^{2\pi} div'\left(\widehat{A}(x')\nabla' u(x',x_N)\right) \cos{(m x_N)}\,dx_N\\
= \frac{1}{\pi} \int_0^{2\pi}\left(g(x',x_N)-u(x',x_N)-u_{NN}(x',x_N)\right)\cos{(m x_N)}\,dx_N.
\end{multline*}
Now we can go on with the same computations already developed in Lemma \ref{equazioni per i coefficienti di Fourier}, with the only difference that 
\[
\frac{1}{\pi} \int_0^{2\pi} \cos{(m x_N)} \,dx_N=0 \quad \text{while} \quad \frac{1}{\pi} \int_0^{2\pi} g(x',x_N)\cos{(m x_N)} \, dx_N= c_m(x').
\]
In the end, we obtain
\[
div'\left(\widehat{A}(x') \nabla' a_m(x')\right) = \left(m^2-1\right) a_m(x') + c_m(x')+\frac{1}{\pi}\left(u_N(x',0)-u_N(x',2\pi\right).
\]
The same procedure gives the equations for $b_m$ and for $a_0$.
\end{proof} 
For a quite general $g$ the study of these equations does not give a complete classification for the possible solutions of \eqref{pb completo}. However, in some particular cases we can obtain sharp results. This will be the object of the following subsections.

\subsection{Inhomogeneous terms independent of $x_N$}

The first generalization concerns a constant $g$. It is straightforward to adapt the arguments of the previous sections, obtaining the following result.
 
\begin{theorem}
Let $N=2$ or $3$. If $g(x',x_N)=\theta \in \R$, one of the following alternatives occurs:
\begin{itemize}
\item[($i$)] if $\theta \ge 0$ there exists a unique solution of \eqref{pb completo} satisfying \eqref{hp su u}. This solution is given by 
\[
u(x',x_N)= \theta(1-\cos x_N).
\]
\item[($ii$)] if $\theta <0$, problem \eqref{pb completo} does not admit any solution satisfying \eqref{hp su u}.
\end{itemize}
\end{theorem}

\bigskip

The next step in the study is to treat the case $g=g(x')$. If we are interested in solutions satisfying \eqref{hp su u} and $g$ is not constant, 
we can show that we do not have such a kind of solution at all. 

\begin{theorem}
Let $N=2$ or $3$, let $g=g(x') \in \mathcal{C}(\R^{N-1})$. If $g$ is not constant, problem \eqref{pb completo} does not admit any solution satisfying \eqref{hp su u}.\\
Equivalently, if there exists $u \in \mathcal{C}^2(\overline{R}_+^N)$ which solves  \eqref{pb completo} and satisfies \eqref{hp su u}, then $g$ is constant.
\end{theorem}
\begin{proof}
Assume that $g$ is not constant; the formal Fourier coefficients of $g$ are 
\[
c_0(x')=2 g(x'), \quad c_m(x')\equiv 0 \quad d_m(x') \equiv 0 \quad \forall m \geq 1.
\]
By contradiction, let $u$ be a solution of \eqref{pb completo} satisfying \eqref{hp su u}. 
Since $c_1 \equiv 0$ and $d_1 \equiv 0$, equations \eqref{eq a_m completo} and \eqref{eq b_m completo} for $m=1$ are
\begin{align*}
div'\left(\widehat{A}(x') \nabla' a_1(x')\right) &= \frac{1}{\pi} \left(u_N(x',0)-u_N(x',2\pi)\right) \\div'\left(\widehat{A}(x') \nabla' b_1(x')\right) &= \frac{1}{\pi} u(x',2\pi).
\end{align*}
Hence we are in position to follow the proof of Lemma \ref{lemma on a_1 and b_1 div}: $a_1$ and $b_1$ are constant, and $u(x',2\pi), u_N(x',0), u_N(x',2\pi)  \equiv 0$ in $\R^{N-1}$. As a consequence, equations \eqref{eq a_m completo} and \eqref{eq b_m completo} become
\begin{align*}
& div'\left(\widehat{A}(x') \nabla' a_m(x')\right) = \left(m^2-1\right) a_m(x')\\
& div'\left(\widehat{A}(x') \nabla' b_m(x')\right) = \left(m^2-1\right) b_m(x').
\end{align*}
Therefore Lemma \ref{coefficienti nulli div} applies: $a_m= b_m \equiv 0$ for every $m \ge 2$, so that
\[
u(x',x_N)= \frac{a_0(x')}{2}+a_1 \cos{x_N} + b_1 \sin{x_N}.
\]
The boundary condition $u(x',0)=0$ implies that $a_0$ is constant, but \eqref{eq a_m completo} for $m=0$ yields
\[
0=div'\left(\widehat{A}(x')\nabla' a_0\right)=2g(x')-a_0,
\]
a contradiction.
\end{proof}

\subsection{A $1$-D inhomogeneous term}

In this subsection we deal with $g=g(x_N)$. In this situation various phenomena may occur. Let us start with :

\paragraph{Non-existence.} 
If $g(x_N)= \sin x_N$, problem \eqref{pb completo} does not admit any solution satisfying \eqref{hp su u}.  This follows from the following general result.

\begin{proposition}\label{d_1 >0}
Let $N=2$ or $3$, let $g \in \mathcal{C}(\overline{\R}_+^N)$ and assume that $d_1 \geq 0$ in $\R^{N-1}$. If there exists a solution $u$ of \eqref{pb completo} such that \eqref{hp su u} holds true, then $d_1 \equiv 0$, $b_1$ is constant,
\[
u(x',2\pi)=0 \quad \text{and} \quad u_N(x',2\pi)=0 \quad \forall x' \in \R^{N-1}.
\]
\end{proposition}
\begin{proof}
Let us consider equation \eqref{eq b_m completo} for $m=1$: since $d_1 \geq 0$ and $u \geq 0$, we have
\[
div'\left(\widehat{A}(x') \nabla' b_1(x')\right) =  d_1+\frac{1}{\pi} u(x',2\pi) \geq 0.
\]
Due to the boundedness of $u$ in the strip $\Sigma$, $b_1$ is bounded in absolute value. Since $N=2$ or $3$, we can apply Theorem \ref{teo Liouville div}, obtaining that $b_1$ is constant, which in turns gives 
\[
u(x',2\pi)=-d_1 \Rightarrow u(x',2\pi)=0=d_1 \qquad \forall x' \in \R^{N-1},
\]
because $u$ is nonnegative. Note that necessarily $u_N(x',2\pi)=0$. 
\end{proof}

\begin{remark}
The previous Proposition applies not only if $g=g(x_N)$. For instance, it gives analogous non-existence results when
\begin{itemize}
\item $g(x',x_N)|_\Sigma$ is decreasing in the $x_N$ direction ($g \neq $ const.) .
\item $g(x',x_N)\geq g(x',2\pi -x_N)$ for every $(x',x_N) \in \R^{N-1} \times (0,\pi)$, with strict inequality in one point.
\end{itemize}
\end{remark}

We have a counterpart of the previous statement which rules out the existence of solutions of \eqref{pb completo} satisfying \eqref{hp su u} when $g(x_N)= \cos x_N$.

\begin{proposition}\label{c_1 >0}
Let $N=2$ or $3$, let $g \in \mathcal{C}(\overline{\R}_+^N)$ and assume that $d_1 \equiv 0$, $c_1 \geq 0$ in $\R^{N-1}$. If there exists a solution $u$ of \eqref{pb completo} such that \eqref{hp su u} holds true, then $c_1 \equiv 0$, $a_1$ is constant, and
\[
u_N(x',0)=0 \qquad \forall x' \in \R^{N-1}.
\]
\end{proposition}
\begin{proof}
In light of Proposition \ref{d_1 >0}, we know that $u_N(x',2\pi)=0$. Moreover, as already observed, from $u(x',0)=0$ and $u \ge 0$ it follows $u_N(x',0) \geq 0$. Thus, considering equation \eqref{eq a_m completo} for $m=1$, we get 
\[
div'\left(\widehat{A}(x') \nabla' a_1(x')\right) =  c_1+\frac{1}{\pi} u_N(x',0) \geq 0,
\]
since $c_1 \geq 0$. The function $a_1$ is bounded in absolute value, hence for Theorem \ref{teo Liouville div} it is constant. Therefore 
\[
c_1+ u_N(x',0) = 0 \Rightarrow u_N(x',0)=0=c_1 \qquad \forall x' \in \R^{N-1}. \qedhere
\]
\end{proof}

\paragraph{Existence and multiplicity.}
For every $N \ge 2$,
\[
u_A(x',x_N)= x_N + A \sin x_N, \qquad A \in [-1,1],
\]
is a one-parameter family of solutions of \eqref{pb completo} with $g(x_N)=x_N$; each $u_A$ satisfies \eqref{hp su u}. Note that in this case 
$c_1=0$ while $d_1<0$, so that the previous Propositions do not apply. Note also that $u_A$ is unbounded in $\R_+^N$ for every $A \in [-1,1]$. 

\paragraph{Existence, uniqueness and $1$-D symmetry.}
For $ m\ge2$, the function $u(x',x_N)=\frac{1}{m^2-1}\left(1-\cos(m x_N)\right)$ is the {\it unique} solution, satisfying \eqref{hp su u}, of problem \eqref{pb completo} for $g(x_N)=\frac{1}{m^2-1}+\cos(m x_N)$. Furthermore, $u$ has 1-D symmetry. The uniqueness result is a consequence of the following general result.

\begin{theorem}\label{theorem pb complete 1D}
Let $N=2$ or $3$, let $g \in \mathcal{C}(\R)$ be such that 
\beq\label{hp su g}
c_1 \ge0 \quad \text{and} \quad d_1 \ge 0,
\eeq
where $c_m$ and $d_m$ are the Fourier coefficients of the function $g$ in $(0,2\pi)$, defined by \eqref{coeff of g}. If there exists $u \in \mathcal{C}^2(\overline{\R}^N_+)$ which solves problem \eqref{pb completo} and satisfies \eqref{hp su u}, then necessarily $c_1=d_1=0$. In this case, the restriction of $u$ to ${\overline{\Sigma}}$ is $1$-dimensional and is uniquely determined as the solution of
\beq\label{1 D pb}
\begin{cases}
-u''(x_N)=u(x_N)-g(x_N) & \text{in $(0,2\pi)$}\\
u(0)=u(2\pi)=0 \\
u'(0)=u'(2\pi) = 0.
\end{cases}
\eeq
In particular, in ${\overline{\Sigma}}$ we have
\begin{multline}\label{u in series}
u(x',x_N)=\frac{c_0}{2}+ \left(-\frac{c_0}{2}+ \sum_{m=2}^{+\infty} \frac{c_m}{m^2-1} \right) \cos x_N +  \left(\sum_{m=2}^{+\infty} \frac{m}{m^2-1}d_m \right) \sin x_N \\
-\sum_{m=2}^{+\infty} \left(\frac{c_m}{m^2-1} \cos (m x_N) + \frac{d_m}{m^2-1} \sin (mx_N) \right).
\end{multline} 
\end{theorem}
\begin{proof}
From Propositions \ref{d_1 >0} and \ref{c_1 >0} we know that, if $u$ exists, then $c_1$ and $d_1$ has to be $0$; in this case $a_1$ and $b_1$ are constant, and $u(x',2\pi), u_N(x',0),$ \\
$u_N(x',2\pi)=0$ in $\R^{N-1}$. Therefore, equations \eqref{eq a_m completo} and \eqref{eq b_m completo} for $m \geq 2$ simplify as
\begin{align*}
div'\left(\widehat{A}(x') \nabla' a_m(x')\right) &=  (m^2-1)a_m(x') + c_m\\
div'\left(\widehat{A}(x') \nabla' b_m(x')\right) &=  (m^2-1) b_m(x')+d_m,
\end{align*}
i.e. 
\begin{align*}
&-div'\left(\widehat{A}(x') \nabla' \left(a_m(x')+\frac{c_m}{m^2-1}\right)\right) + (m^2-1) \left(a_m(x')+\frac{c_m}{m^2-1}\right)=0\\
&-div'\left(\widehat{A}(x') \nabla' \left(b_m(x')+\frac{d_m}{m^2-1}\right)\right) + (m^2-1) \left(b_m(x')+\frac{d_m}{m^2-1}\right)=0.
\end{align*}
We can apply Lemma \ref{coefficienti nulli div}, obtaining
\begin{align*}
&a_m(x')=a_m= -\frac{c_m}{m^2-1} \qquad \forall m \geq 2\\
&b_m(x')=b_m= -\frac{d_m}{m^2-1} \qquad \forall m \geq 2.
\end{align*}
Now, let us consider in $\Sigma$
\begin{multline*}
\frac{a_0(x')}{2}+ a_1 \cos x_N + b_1 \sin x_N \\
-\sum_{m=2}^{+\infty} \left(\frac{c_m}{m^2-1} \cos (mx_N) + \frac{d_m}{m^2-1} \sin (mx_N) \right).
\end{multline*}
It is a series of $\mathcal{C}^\infty$ functions which is convergent together with the series of the derivates w.r.t. $x_N$, since the sequences $\{c_m\} $ and $\{d_m\}$ belong to $\mathit{l}^2$. In $\Sigma$ the series is equal to $u$, and the equality can be extended up to the boundary since both the series itself and $u$ are $\mathcal{C}^1(\bar \Sigma)$. We also know that $u(x',0)=0$ and $u_N(x',0)=0$. Using the Dirichlet boundary condition we get that $a_0$ is constant too, and in particular equation \eqref{eq a_m completo} for $m=0$ implies $a_0=c_0$. Now, from the "initial" conditions we get the expression of $a_1$ and $b_1$. \\
To sum up, we proved that $u|_{\Sigma}$ is $1$-D, thus a solution of 
\[
-u''(x_N)=u(x_N)-g(x_N) \qquad \text{for } x_N \in (0,2\pi)
\]
with the boundary conditions stated in \eqref{1 D pb}.
\end{proof}
As an immediate consequence we obtain
\begin{theorem}\label{theorem pb complete 1D 2}
Let $N=2$ or $3$, let $g \in \mathcal{C}(\R)$ be a $2\pi$-periodic function satisfying \eqref{hp su g}, where $c_m$ and $d_m$ are the Fourier coefficients of the function $g$. If there exists $u \in \mathcal{C}^2(\overline{\R}^N_+)$ which solves problem \eqref{pb completo} and satisfies \eqref{hp su u}, then necessarily $c_1=d_1=0$. In this case, $u$ is $1$-dimensional, $2\pi$-periodic and it is uniquely determined as the solution of
\[
\begin{cases}
-u''(x_N)=u(x_N)-g(x_N) & \text{in $(0,+\infty)$}\\
u(0)=u(2\pi)=0 \\
u'(0)=u'(2\pi) = 0.
\end{cases}
\]
The expression of $u$ in Fourier series is given by \eqref{u in series}.
\end{theorem}

In view of the example with $g(x_N)= x_N$ ($c_1=0$, $d_1<0$) we see that the assumptions $c_1 \ge 0$ and $d_1 \ge 0$ are necessary for Theorem \ref{theorem pb complete 1D} and Theorem \ref{theorem pb complete 1D 2}. We also remark that the non-negativity of both $c_1$ and $d_1$ is not sufficient to guarantee the existence of a solution of \eqref{pb completo} which satisfies \eqref{hp su u}. Indeed, as an immediate consequence of Proposition \ref{teo 10} (proved in the next subsection),  we have non-existence of solutions of \eqref{pb completo} satisfying \eqref{hp su u} in case 
\[
g(x_N)= C_1 \sin(m x_N) \quad \text{or} \quad g(x_N)= C_2 \cos(m x_N) \quad m \ge 2, C_1,C_2 \in \R.
\]
Note that $c_1 = d_1 =0$ in the above examples.

Another class of functions $g$ for which there is non existence is considered in the next result, of independent interest, 

\begin{proposition}\label{max prin}
Let $N \ge 2$. If $g \leq 0$ and it is non-constant, then a nonnegative solution of \eqref{pb completo} has to be positive. In particular, if $N=2,3$ and  $d_1 \ge 0$, then problem \eqref{pb completo} does not admit any solution satisfying \eqref{hp su u}. 
\end{proposition}
\begin{proof} 
By the strong maximum principle $u$ must be positive in $\R^N_+$. Since $d_1\ge 0$, if a solution $u$ existed, from Proposition \ref{d_1 >0} it should satisfy $u(x',2\pi)=0$ for every $x' \in \R^{N-1}$. A contradiction.
\end{proof}

A typical example is given by the function $g(x_N)=-\theta-\cos x_N$, with $\theta \ge 1$. Note that $c_1<0$ and $d_1=0$ in this example.  


\medskip

\subsection{General inhomogeneous terms}

In this subsection we will consider $g$-s depending on both $x'$ and $x_N$. As before, we will denote by $c_m$ and $d_m$ the Fourier coefficient of the $2\pi$-periodic extension of $x_N \in (0,2\pi)\mapsto g(x',x_N)$.

\medskip 

We have always begun our analysis trying to prove that 
\beq\label{passo base}
u(x',2\pi) \equiv 0 \quad \text{and} \quad u_N(x',0) \equiv 0 \quad \text{in $\R^{N-1}$}.
\eeq 
As a consequence, equations \eqref{eq a_m completo} and \eqref{eq b_m completo} are considerably simplified, since all the boundary terms have to vanish identically:
\begin{align*}
& div'\left(\widehat{A}(x') \nabla' a_m(x')\right) = \left(m^2-1\right) a_m(x') + c_m(x')\\
& div'\left(\widehat{A}(x') \nabla' b_m(x')\right) = \left(m^2-1\right) b_m(x') + d_m(x').
\end{align*}
We have already observed that, if $N=2$ or $3$, sufficient conditions in order to obtain \eqref{passo base} are $c_1 \ge 0$ and $d_1 \ge 0$.\\
In general (for every $N \ge 2$), assume that \eqref{passo base} holds true. Assume also that there exists $\bar m \geq 2$ such that $c_{\bar m} \equiv 0$. Then
\[
div'\left(\widehat{A}(x') \nabla' a_{\bar m}(x')\right) = \left(\bar m^2-1\right) a_{\bar m}(x'),
\]
which is of type \eqref{eq per coefficienti div} with $\lambda > 0$. From Lemma \ref{coefficienti nulli div} it follows $a_{\bar m} \equiv 0$. The same holds true for every $b_{\bar m}$ such that $d_{\bar m} \equiv 0$. We point out that this is true even for $N>3$.
\begin{proposition}\label{sui coeff nulli}
Let $N \ge 2$, let $g \in \mathcal{C}(\overline{\R}_+^N)$, let $u$ be a solution of \eqref{pb completo} satisfying \eqref{hp su u}, and let $a_m$ and $b_m$ be its formal Fourier coefficients in $\Sigma$ defined by \eqref{coefficienti}; assume that \eqref{passo base} holds true. Then for every $m \geq 2$ such that $c_m \equiv 0$ it results $a_m \equiv 0$, and for every $m \geq 2$ such that $d_m \equiv 0$ it results $b_m \equiv 0$.
\end{proposition}

As far as the coefficient $a_0$ is concerned, we have a similar result, but only in low dimension.

\begin{proposition}\label{su a_0 nullo}
Let $N=2$ or $3$, let $g \in\mathcal{C}(\overline{\R}_+^N)$, let $u$ be a solution of \eqref{pb completo} satisfying \eqref{hp su u}, and let $a_m$ and $b_m$ be its formal Fourier coefficients defined by \eqref{coefficienti}; assume that \eqref{passo base} holds true. If $c_0 \le 0$, then $a_0 = c_0 \equiv 0$.
\end{proposition}
\begin{proof}
Since $u \geq 0$, equation \eqref{eq a_m completo} for $m=0$ is
\[
div'\left(\widehat{A}(x') \nabla' a_0(x')\right) = - a_0(x') + c_0(x') \leq 0 \qquad \text{in $\R^{N-1}$}.
\]
Since $a_0$ is bounded and $N=2$ or $3$, for Theorem \ref{teo Liouville div} $a_0$ is constant. But then 
\[
0= div'\left(\widehat{A}(x') \nabla' a_0(x')\right) = - a_0 + c_0. 
 \]
 Thus, $ 0 \le a_0 = c_0 \le 0$. 
\end{proof}

\bigskip

In what follows we first consider 
\[
g(x',x_N)=f(x')\varphi(x_N) \in \mathcal{C}(\overline{\R}^N_+), \qquad g \not\equiv 0.
\]
In the expansion of the $2\pi$-periodic extension of $x_N \in (0,2\pi) \mapsto g(x',x_N)$, the Fourier coefficients are
\[
c_m(x')=f(x')\gamma_m \quad \forall m \geq 0, \qquad
 d_m(x')=f(x')\delta_m \quad \forall m \geq 1,
\]
where $\gamma_m$ and $\delta_m$ are the (constant) Fourier coefficients of the $2\pi$-periodic extension of $x_N \in (0,2\pi) \mapsto \varphi(x_N)$.

\begin{remark}
Let $N=2$ or $3$. In light of Propositions \ref{d_1 >0} and \ref{c_1 >0}, we know that if $f(x') \delta_1 \geq 0$, $f(x') \delta_1 \neq 0$, then there are no solutions of \eqref{pb completo} satisfying \eqref{hp su u}. The same holds true if $f(x') \gamma_1 \geq 0$, $f (x')\gamma_1\neq 0$ and $\delta_1=0$.
\end{remark}

If $N=2$ or $3$ and $\gamma_1=\delta_1=0$, from Propositions \ref{d_1 >0} and \ref{c_1 >0} we know that $a_1$ and $b_1$ are constant and \eqref{passo base} holds true. Hence, equations \eqref{eq a_m completo} and \eqref{eq b_m completo} simplify as
\begin{align*}
& div'\left(\widehat{A}(x')\nabla' a_m\right)= (m^2-1)a_m + f(x')\gamma_m \qquad \forall m \neq 1,\\
& div'\left(\widehat{A}(x')\nabla' b_m\right)= (m^2-1)b_m + f(x')\delta_m \qquad \forall m \ge 2.
\end{align*}
It is not difficult to obtain the following non-existence result.

\begin{proposition}\label{teo 10}
Let $N=2$ or $3$, let $g \in \mathcal{C}(\overline{\R}^N_+)$ and assume 
\[
g(x',x_N)=f(x') \cos(m x_N) \quad \text{or} \quad g(x',x_N)=f(x') \sin(m x_N),
\]
where $m \geq 2$ and $f$ is not identically $0$. Then there are no solutions of \eqref{pb completo} satisfying \eqref{hp su u}.
\end{proposition}
\begin{proof}
Let us first consider the case $g(x',x_N)=f(x') \sin(m x_N)$. By contradiction, let $u$ be a solution of \eqref{pb completo} satisfying \eqref{hp su u}. Applying Propositions \ref{d_1 >0}, \ref{c_1 >0}, \ref{sui coeff nulli} and \ref{su a_0 nullo}, we obtain the following particular form for $u$ in $\Sigma$:
\[
u(x',x_N)= a_1 \cos{x_N} + b_1 \sin{x_N}+ b_{m}(x') \sin{(m x_N)}.
\]
Since $u_N(x',0) \equiv 0$, we deduce that $b_m = - \frac{b_1}{m}$ is constant. On the other hand it is a solution of
\[
0=div'\left(\widehat{A}(x')\nabla' b_m\right)=(m^2-1)b_m+f(x');
\]
thus $f$ must be constant, that is $f(x')=f \equiv \theta \in \R\setminus \{0\}$. But in this case, imposing the initial condition $u(x',0)=0$ we would obtain $a_1=0$ and consequently
\[
u(x',x_N)= \frac{\theta }{m^2-1} \left( m \sin x_N -\sin(m x_N)\right),
\]
which does not satisfy \eqref{hp su u} because it assumes negative values (it is odd, $2\pi$-periodic and not identically zero).

When $g(x',x_N)=f(x') \cos(m x_N)$, we can argue as before to find  that $u$ has the form :
\[
u(x',x_N)= a_1 \cos{x_N} + b_1 \sin{x_N}+ a_{m}(x') \cos{(m x_N)}.
\]
The boundary condition $u_N(x',0) =0$ implies $ b_1 =0$, while we get $ a_m(x') = -a_1 = const. $ from $u(x',0)=0$.  Hence $0 = (m^2 -1) a_m + f(x') $ and so $f(x')=f \equiv \theta \in \R\setminus \{0\}$. Finally $u$ has the form 
\[
u(x',x_N)= \frac{\theta }{m^2-1} \left(\cos (x_N)-\cos(m x_N)\right).
\]
Observe that $ 0 \le u(x', \frac{2\pi}{m}) = \frac{\theta }{m^2-1}\left(\cos (\frac{2\pi}{m})-1\right)$ implies $ \theta \le 0$, while $ 0 \le u(x', \frac{\pi}{m}) = \frac{\theta }{m^2-1}\left(\cos (\frac{\pi}{m})+1\right)$ yields $\theta \ge 0$. A contradiction. 
\end{proof}

More in general, the same proof yields  
\begin{proposition}\label{teo 11}
Let $N=2$ or $3$, let $g \in \mathcal{C}(\overline{\R}^N_+)$ and assume 
\begin{multline*}
g(x',x_N)=\frac{c_0(x')}{2}+ \sum_{m\in I_1} c_{m}(x') \cos(m x_N)+ d_{\bar n}(x') \sin( \bar n x_N) \\
\text{or } \quad g(x',x_N)=c_{ \bar m}(x') \cos( \bar m x_N)+ \sum_{n \in I_2} d_{n}(x') \sin(n x_N),
\end{multline*}
where $I_1, I_2 \subset (\N\setminus \{0,1\})$ are finite sets, $\bar n \geq 2$, $\bar m \in (\N \setminus \{1\})$ and $d_{\bar n}$, $c_{\bar m}$ are not identically constant. Then there are no solutions of \eqref{pb completo} satisfying \eqref{hp su u}.
\end{proposition}

\bigskip

In what follows we set $N=2$ or $3$ and we show that it is possible to use the method of the Fourier coefficients in order to obtain a complete classification when $c_1=d_1=0$ and only a finite number of the Fourier coefficients of $g$ are not identically zero. 

Let
\beq\label{linearity 2}
g(x',x_N)=\frac{c_0(x')}{2}+ \sum_{m\in I_1} c_{m}(x') \cos(m x_N)+ \sum_{n \in I_2 }d_{n}(x') \sin(n x_N),
\eeq
where $I_1=\{m_1, \ldots, m_{k_1}\}, I_2=\{n_1, \ldots, n_{k_2}\} \subset (\N\setminus \{0,1\})$. As far as $c_0$ is concerned, it can be identically $0$ or not. Only to fix our minds, we assume $c_0(x') \neq 0$; furthermore, for the sake of simplicity, we suppose that $c_0,c_{m_j}, d_{n_j} \in \mathcal{C}^{\infty}(\R^{N-1})$. \\
In what follows we will show that, \emph{if there exists $u \in \mathcal{C}^2(\overline{\R}_+^N)$ which solves \eqref{pb completo} for this particular $g$ and satisfies \eqref{hp su u}, then we can determine the explicit expression of $u$}.\\
Note that, since $c_1=d_1=0$, Propositions \ref{d_1 >0} and \ref{c_1 >0} imply that $a_1$ and $b_1$ are constant, and \eqref{passo base} holds true; thus, by Proposition \ref{sui coeff nulli} we obtain
\begin{multline*}
u(x',x_N)= \frac{a_0(x')}{2} + a_1\cos{x_N}+b_1 \sin{x_N}\\
+\sum_{j=1}^{k_1} a_{m_j}(x') \cos(m_j x_N) + \sum_{j=1}^{k_2} b_{n_j}(x') \sin(n_j x_N),
\end{multline*}
in $\Sigma$, where $a_0, a_{m_j}$ and $b_{n_j}$ are solutions of
\begin{align}
div'\left(\widehat{A}(x')\nabla' a_0(x')\right) &=-a_0(x') + c_0(x') \label{eq for a_0}\\
div'\left(\widehat{A}(x')\nabla' a_{m_j}(x')\right) &=(m_j^2-1)a_{m_j}(x') + c_{m_j}(x') \label{eq for a_{m_j}}\\
div'\left(\widehat{A}(x')\nabla' b_{n_j}(x')\right) &=(n_j^2-1)b_{n_j}(x') + d_{n_j}(x') \label{eq for b_{n_j}}.
\end{align}
Propositions \ref{teo 10} and \ref{teo 11} imply that, if there exists a unique $m \in \N \setminus \{1\}$ such that $c_m \neq 0$ and is not constant, or if there exists a unique $m \ge 2$ such that $d_m \neq 0$ and is not constant, then a solution of \eqref{pb completo} satisfying \eqref{hp su u} does not exist. If we are not in this situation and such a solution exists, this system of PDEs (or ODEs if $N=2$), together with the boundary conditions $u(x',0)=0$ and $u_N(x',0)=0$ permits to deduce the explicit expression of $a_0, a_{m_j}$ and $b_{n_j}$. We start observing that the boundary condition $u(x',0)=0$ involves only $a_0$ and $a_{m_j}$, while $u_N(x',0)=0$ involves the $b_{n_j}$. Thus, we can consider the system of $k_1+2$ equations given by $u(x',0)=0$ together with \eqref{eq for a_0} and \eqref{eq for a_{m_j}}; the unknowns are the functions $a_0$ and $a_{m_j}$, while we consider $a_1$ as a parameter; from $u(x',0)=0$ we get 
\beq\label{eq201*}
a_{0}(x')=-2a_1-2\sum_{j=1}^{k_1} a_{m_j}(x');
\eeq
As a consequence
\[
div'\left(\widehat{A}(x')\nabla' a_0(x')\right) = -2 \sum_{j=1}^{k_1} div'\left(\widehat{A}(x')\nabla' a_{m_j}(x')\right), 
\]
and
\[
-a_0(x') + c_0(x')=2a_1+ 2\sum_{j=1}^{k_1} a_{m_j}(x')+c_0(x'),
\]
so that equation \eqref{eq for a_0} gives
\[
\sum_{j=1}^{k_1} div'\left(\widehat{A}(x')\nabla' a_{m_j}(x')\right) = -a_1 -\sum_{j=1}^{k_1} a_{m_j}(x')-\frac{c_0(x')}{2}.
\]
We plug \eqref{eq for a_{m_j}} for $j \ge 1$ on the left hand side:
\[
 \sum_{j=1}^{k_1} \left[(m_j^2-1)a_{m_j}(x')+c_{m_j}(x')\right] = -a_1 -\sum_{j=1}^{k_1} a_{m_j}(x')-\frac{c_0(x')}{2},
 \]
i.e.
\beq\label{eq201}
a_{m_1}(x') = \frac{1}{m_1^2} \left[ -a_1 -f(x') -\sum_{j=2}^{k_1}m_j^2 a_{m_j}(x')\right],
\eeq
where $f(x')=c_0(x')/2+\sum_{j=1}^{k_1} c_{m_j}(x')$. Note that now equation \eqref{eq201} together with \eqref{eq for a_{m_j}} for $j \ge 2$ is a system of $k_1+1$ equations in the unknowns $a_{m_j}$ but without $a_0$. If we can solve it, we can recover $a_0$ using the \eqref{eq201*}.\\
We can iterate the same argument: from \eqref{eq201} we have 
\begin{multline*}
div'\left(\widehat{A}(x')\nabla' a_{m_1}(x')\right)\\ =
\frac{1}{m_1^2} \left[ - div'\left(\widehat{A}(x')\nabla' f(x')\right)-\sum_{j=2}^{k_1} m_j^2 div'\left(\widehat{A}(x')\nabla' a_{m_j}(x')\right)  \right]
\end{multline*}
(this is why we required the $c_{m_j}$ smooth) and
\begin{multline*}
(m_1^2-1)a_{m_1}(x')+ c_{m_1}(x') \\= \frac{m_1^2- 1}{m_1^2} \left[ -a_1 -f(x') -\sum_{j=2}^{k_1}m_j^2 a_{m_j}(x')\right] + c_{m_1}(x');
\end{multline*}
equation \eqref{eq for a_{m_j}} for $j=1$ gives
\begin{multline*}
- div'\left(\widehat{A}(x')\nabla' f(x')\right)-\sum_{j=2}^{k_1} m_j^2 div'\left(\widehat{A}(x')\nabla' a_{m_j}(x')\right) \\
= -(m_1^2- 1) a_1 -(m_1^2-1) f(x')- (m_1^2-1)\sum_{j=2}^{k_1}m_j^2 a_{m_j}(x') + m_1^2 c_{m_1}(x'),
\end{multline*}
i.e.
\begin{multline}\label{eq202}
a_{m_2}(x')=\frac{1}{m_2^2(m_2^2-m_1^2)} \\
\cdot \left[(m_1^2-1)a_1 -f_1(x') - \sum_{j=3}^{k_1} m_j^2(m_j^2-m_1^2) a_{m_j}(x')\right],
\end{multline}
where 
\[
f_1(x')= - div'\left(\widehat{A}(x')\nabla' f(x')\right)-\sum_{j=2}^{k_1} m_j^2 c_{m_j}(x')-c_{m_1}.
\]
Equation \eqref{eq202} together with \eqref{eq for a_{m_j}} for $j \ge 2$ is a system of $k_1$ equations in the unknowns $a_{m_j}$ for $j \ge 2$, but without $a_0$ and $a_{m_1}$.  If we can solve it, we can recover $a_{m_1}$ using the \eqref{eq201}, and then $a_0$ using the \eqref{eq201*}. \\
Iterating the procedure $k_1+2$ times (here we have to assume $k_1$ finite), we obtain $a_{m_{k_1}}$ as function of the Fourier coefficients of the $g$ (note that the more $k_1$ is large the more we have to require the $c_{m_j}$-s smooth), and successively the others $a_{m_j}$. Note that $a_0$ and $a_{m_j}$ are functions of $a_1$.\\
The same procedure works for the coefficients $b_{n_j}$-s, starting from $u_N(x',0)=0$. In the end we get the explicit expression of $u$ in function of the two ``parameters'' $a_1$ and $b_1$. At this point it is sufficient to impose that $u$ solves the considered differential equation 
to determine $a_1$ and $b_1$.

Let us see the iterative procedure in action with an example: let $N=2$ and
\begin{multline}
g(x,y)=\left( \frac{2}{(1+x^2)^2}-4 \frac{x}{(1+x^2)^2}\arctan x
+ (\arctan x)^2\right) \\+ \left(-\frac{2}{(1+x^2)^2} +4 \frac{x}{(1+x^2)^2}\arctan x + 3 (\arctan x)^2\right) \cos(2y)\\
=\frac{c_0(x)}{2} + c_2(x) \cos(2y).
\end{multline}

\begin{proposition}
There is a unique solution of
\beq\label{pb ex}
\begin{cases}
-\Delta u=u-g & \text{in $\R_+^N$}\\
u(x,0)=0 \\
\text{$u$ satisfies \eqref{hp su u}},
\end{cases}
\eeq
whose explicit expression is 
\[
u(x,y)=(\arctan x)^2\left(1-\cos(2y)\right). 
\]
\end{proposition}
\begin{proof}
Due to the form of $g$ we know that if $u$ solves \eqref{pb ex} then
\beq\label{exp of u}
u(x,y)= \frac{a_0(x)}{2} + a_1 \cos y + b_1 \sin y + a_2(x) \cos(2y),
\eeq	
with $u_y(x,0)=0$ (Lemma \ref{lemma on a_1 and b_1 div} and Proposition \ref{sui coeff nulli}). Thus $b_1=0$. As far as $a_0$ and $a_2$ is concerned, they solve
\begin{align}
& a_0''(x)=-a_0(x)+c_0(x) \label{eq for a_0 ex}\\
& a_2''(x)= 3a_2(x)+c_2(x) \label{eq for a_2 ex}.
\end{align}
From $u(x,0)=0$ we deduce
\beq\label{eq203}
a_0(x)=-2a_1-2a_2(x).
\eeq
Hence \eqref{eq for a_0 ex} gives
\[
a_2''(x) = -a_1 -a_2(x) -\frac{c_0(x)}{2};
\]
we plug \eqref{eq for a_2 ex} on the left hand side, obtaining 
\beq\label{a_2 from c}
a_2(x)= -\frac{c_0(x)}{8}-\frac{c_2(x)}{4}-\frac{a_1}{4} = -(\arctan x)^2-\frac{a_1}{4},
\eeq
and consequently from \eqref{eq203}
\beq\label{a_0 from c}
a_0(x)= \frac{c_0(x)}{4}+\frac{c_2(x)}{2}-\frac{3}{2}a_1=2\left(\arctan x\right)^2-\frac{3}{2}a_1.
\eeq
Note that it is sufficient to substitute the explicit expressions of $c_0$ and $c_2$ (which are given by $g$) in order to get $a_0$ and $a_2$, and no integration is required.\\
So far, we proved that a solution of \eqref{pb ex} is of type
\begin{align*}
u_{a_1}(x,y) &=(\arctan x)^2\left(1-\cos(2y)\right)  -\frac{a_1}{4}\left(3 - 4\cos y +\cos(2y)\right) \\ &= (\arctan x)^2\left(1-\cos(2y)\right) -\frac{a_1}{2}\left(1 - \cos y \right)^2, 
\end{align*}
which is non negative if and only if $a_1 \le 0$. It is straightforward to check that $u_{a_1}$ solves \eqref{pb ex} only if $a_1=0$.
\end{proof}

\begin{remark} For a generic $g$ of the form \eqref{linearity 2}, the iterative procedure we introduced above can be used as a test in order to check if \eqref{pb ex} has at least one solution satisfying \eqref{hp su u}.

For instance it is immediate to check that \eqref{pb ex} with
\[
g(x,y)= \cos(2x) + \sin(3x)\cos(2y) ,
\]
has not a solution satisfying \eqref{hp su u}. Indeed, if such a solution existed, then its explicit expression would be \eqref{exp of u} with $a_0$ and $a_2$ given by 
\[
a_0(x)=\frac{c_0(x)}{4}+\frac{c_2(x)}{2}-\frac{3}{2} a_1 \quad a_2(x)=-\frac{c_0(x)}{8}-\frac{c_2(x)}{4}-\frac{a_1}{4},
\]
(cf. \eqref{a_0 from c} and \eqref{a_2 from c}) where $c_0(x)=2\cos(2x)$ and $c_2(x)=\sin(3x)$;
but $\frac{a_0(x)}{2} + a_1 \cos(y) + a_2(x) \cos(2y)$ is not a solution of $-\Delta u=u-g$.

Last but not least, we also remark that if $ \lambda_1,..., \lambda_k$ are nonnegative real numbers and $u_1,...,u_k$ are solutions of \eqref{pb ex} with $g=g_j$, $\, j=1,...,k$, then the function $ u = \sum_{j=1}^{k} \lambda_j u_j$ is a solution of \eqref{pb ex} with $g= \sum_{j=1}^{k} \lambda_j g_j$.
Thus, combining in a suitable way the examples considered before, we can construct many other functions $g$ for which we have existence and uniqueness of the solution or existence and multiplicity of the solutions. 
\end{remark}

\end{document}